\documentclass[draft]{amsart}
\usepackage{
color,tikz}

\newtheorem{theorem}{Theorem}
\theoremstyle{definition}
\newtheorem{definition}{Definition}
\newtheorem{remark}{Remark}
\newtheorem{example}{Example}
\newtheorem{conjecture}{Conjecture}

\newcommand{\R}{\mathbb{R}}
\newcommand{\Z}{\mathbb{Z}}
\newcommand{\K}{\mathbb{K}}
\newcommand{\C}{\mathbb{C}}
\newcommand{\Q}{\mathbb{Q}}

\begin{document}

\title{Invariants of long knots}
\author{Rinat Kashaev}
\address{Section de math\'ematiques, Universit\'e de Gen\`eve,
2-4 rue du Li\`evre, 1211 Gen\`eve 4, Suisse\\}
\email{rinat.kashaev@unige.ch}
\date{December 31, 2019}
 \dedicatory{Dedicated to Nikolai Reshetikhin on the occasion of his 60th birthday }
\begin{abstract}
 By using the notion of a rigid R-matrix in a monoidal category and the Reshetikhin--Turaev functor on the category of tangles, we review the definition of the associated invariant of long knots. In the framework of the monoidal categories of relations and spans over sets, by introducing racks associated with pointed groups, we illustrate the construction and the importance of consideration of long knots.  Else, by using the restricted dual of algebras and  Drinfeld's quantum double construction, we show that to any Hopf algebra $H$ with invertible antipode, one can associate a universal long knot invariant $Z_H(K)$ taking its values in the convolution algebra  $((D(H))^o)^*$ of the restricted dual Hopf algebra $(D(H))^o$ of the quantum double $D(H)$ of $H$. That extends the known constructions of universal invariants previously  considered mostly either in the case of finite dimensional Hopf algebras or by using some topological completions.
\end{abstract}
\maketitle
\section*{Introduction}

This paper is to a large extent a review of certain aspects of quantum invariants where we restrict ourselves exclusively to the context of long knots. More generally, one can consider also the string links, but it is known that the topological classes of string links are in bijection with the classes of ordinary links only if the number of components is one, i.e. if a string link is a long knot. 

We describe in detail the construction of invariants of long knots by using rigid R-matrices (solutions of the quantum Yang--Baxter relation) in monoidal categories. The importance of long knots (as opposed to usual closed knots) is illustrated by considering a general class of  group-theoretical R-matrices put into the context of monoidal categories of relations and spans over sets. These R-matrices are indexed by pointed groups that is groups with a distinguished element. The underlying racks seem not to be considered previously in the existing literature. 

Drinfeld's quantum double construction gives rise to a large class of rigid R-matrices, and the associated invariants get factorized through universal invariants associated to underlying Hopf algebras.
Such universal invariants were introduced and studied in a number of works~\cite{MR1025161,MR1124415,MR1153694,MR1227011,MR1324033, MR2186115,MR2253443,MR2251160} mostly either in the context of finite dimensional Hopf algebras or certain topological completions, for example by considering formal power series. 
Here, we  define the universal invariants purely algebraically and with minimal assumptions on the underlying Hopf algebras. In particular, we emphasize the case of infinite dimensional Hopf algebras. The distinguishing feature of our approach is the use of the restricted or finite dual of an algebra in conjunction with the quantum double construction.

The outline of the paper is as follows.

In section~\ref{sec:lk} we recall the definitions of long knots and their diagrams and introduce the notions of a normal diagram and a normalization of an arbitrary diagram.

In section~\ref{sec:ilkrrm} we recall the definition of a rigid R-matrix in a monoidal category and give a detailed description of a long knot invariant associated to a given rigid R-matrix.

In section~\ref{sec:rrmfr} we consider a special class of rigid R-matrices in the categories of relations and spans over sets. Each such R-matrix is associated to a pointed group with a canonical structure of a rack, and Theorem~\ref{thm:1}  identifies the associated invariant with the set of representations of the knot group into the group that underlies the rack. The example of an extended Heisenberg group gives rise to an invariant ideal in the polynomial algebra $\Q[t,t^{-1},s]$ closely related but not equivalent to the Alexander polynomial, at least if the latter admits higher multiplicity roots.

In section~\ref{sec:ifha}, based on the restricted dual of an algebra and Drinfeld's quantum double construction, we describe a universal invariant associated to any Hopf algebra with invertible antipode.

\subsection*{Acknowledgements} I would like to thank  Bruce Bartlett, L\'eo B\'enard, Joan Porti, Louis-HadrienRobert, Arkady Vaintrob, Roland van der Veen and Alexis Virelizier for valuable discussions.
\section{Long knots}\label{sec:lk}
\begin{definition}
An embedding $f\colon \R\to\R^3$ is called \emph{\color{blue}  long knot } if there exist  $a,b\in\R$ such that $f(t)=(0,0,t)$ for any $t<a$ or $t>b$.

Two long knots $f,g\colon \R\to\R^3$ are called  \emph{\color{blue}  equivalent}  if they are ambient isotopic, that is if there exists an ambient isotopy 
\begin{equation}
H\colon \R^3\times[0,1]\to  \R^3\times[0,1],\quad H(x,t)=(h_t(x),t),\quad h_0=\operatorname{id}_{\R^3},
\end{equation}
such that, for any $t\in[0,1]$, $h_t\circ f$ is a long knot and $g=h_1\circ f$. 

A long knot is called \emph{\color{blue}  tame or regular} if it is equivalent to a polygonal long knot.

An (oriented)  \emph{\color{blue}  long knot diagram} $D$ is a (1,1)-tangle diagram in $\R^2$ representing a tame long knot
\begin{equation}
D= \begin{tikzpicture}[baseline=-3]
 \node[draw] (a) at (0,0) {$D$};
 \draw[thick,->] (0,-.5)--(a)--(0,.5);
\end{tikzpicture}\ .
\end{equation}

Two long knot diagrams are called  \emph{\color{blue} (Reidemeister) equivalent} if they can be related to each other by a finite sequence of oriented Reidemeister moves of all types.
\end{definition}
\begin{remark}
The set of long knot diagrams is a monoid with respect to the composition 
\begin{equation}
D\circ D':=\ 
 \begin{tikzpicture}[baseline=8]
 \node[draw] (a) at (0,.7) {$D$};
  \node[draw] (b) at (0,0) {$D'$};
 \draw[thick,->] (0,-.5)--(b)--(a)--(0,1.2);
\end{tikzpicture}
\end{equation}

\end{remark}
\begin{remark} In the case of long knots, the Reidemeister theorem states that two long knot diagrams are equivalent if and only if the corresponding long knots are equivalent. Furthermore, 
a folklore theorem states that the natural map of long knot diagrams to closed oriented knot diagrams
 \begin{equation}
 \begin{tikzpicture}[baseline=-2]
 \node[draw] (a) at (0,0) {$D$};
 \draw[thick,->] (0,-.5)--(a)--(0,.5);
\end{tikzpicture}\ 
\mapsto
\ 
 \begin{tikzpicture}[baseline=-2]
 \coordinate (b) at (-.7,0);
 \node[draw] (a) at (0,0) {$D$};
 \draw[thick,->] (a) to [out=90,in=90] (b);
  \draw[thick] (b) to [out=-90,in=-90] (a);
\end{tikzpicture}
\end{equation}
induces a bijection between the respective Reidemeister equivalence classes. In particular, any invariant of long knots is also an invariant of closed knots.
\end{remark}

In what follows, we will always assume that a long knot diagram is put into a generic position with respect to the vertical axis so that all crossing  have non-vertical strands as in the letter X. We denote by $w(D)$ the writhe of $D$ defined as the number of positive crossings minus the number of negative crossings where a crossing is called positive if the ordered pair of tangent vectors $(v_{\text{overpass}},v_{\text{underpass}})$ induces the standard orientation of $\R^2$.
\begin{definition}
 A (long knot) diagram is called \emph{\color{blue} normal} if it has no local extrema (with respect to vertical direction) oriented from left to right like 
  \begin{tikzpicture}[baseline, scale=0.2]
 \draw[thick] (0,0) to [out=90,in=180] (1,1);
 \draw[thick,->] (1,1) to [out=0,in=90] (2,0);
\end{tikzpicture}\,
  and 
  \begin{tikzpicture}[baseline=-5, scale=0.2]
 \draw[thick] (0,0) to [out=-90,in=180] (1,-1);
 \draw[thick,->] (1,-1) to [out=0,in=-90] (2,0);
\end{tikzpicture} . 

To any diagram $D$, we associate its \emph{\color{blue} normalization} $\dot{D}$, the diagram obtained from $D$ by the replacements
\begin{equation}
  \begin{tikzpicture}[baseline=2, scale=0.3]
 \draw[thick] (0,0) to [out=90,in=180] (1,1);
 \draw[thick,->] (1,1) to [out=0,in=90] (2,0);
\end{tikzpicture}
\ \mapsto
\begin{tikzpicture}[baseline=5,xscale=.3,yscale=0.2]
 \coordinate (a0) at (0,0);
 \coordinate (a1) at (1,3);
 \coordinate (a2) at (2,0);
\draw[thick] (a1) to [out=180,in=135]  (a2);
\draw[line width=5, color=white] (a0) to [out=45,in=0]  (a1);
\draw[thick,->] (a0) to [out=45,in=0]  (a1);
\end{tikzpicture},
\quad
  \begin{tikzpicture}[baseline=-7, scale=0.3]
 \draw[thick] (0,0) to [out=-90,in=180] (1,-1);
 \draw[thick,->] (1,-1) to [out=0,in=-90] (2,0);
\end{tikzpicture}\ 
\mapsto
\begin{tikzpicture}[baseline=-12,xscale=.3,yscale=0.2]
 \coordinate (a0) at (0,0);
 \coordinate (a1) at (1,-3);
 \coordinate (a2) at (2,0);
\draw[thick] (a1) to [out=180,in=-135]  (a2);
\draw[line width=5, color=white] (a0) to [out=-45,in=0]  (a1);
\draw[thick,->] (a0) to [out=-45,in=0]  (a1);
\end{tikzpicture}.
\end{equation}
\end{definition}
It will be of special interest for us the normal long knot diagrams 
\begin{equation}\label{eq:xi+-}
 \xi^+:=
\begin{tikzpicture}[baseline=15,scale=.5]
 \coordinate (a0) at (0,0);
 \coordinate (a1) at (1,2);
 \coordinate (a2) at (0,2);
 \coordinate (a3) at (2,.5);
\coordinate (a4) at (1,.5);
\coordinate (a5) at (2,2.5);
\draw[thick] (a0) to [out=90,in=-90]  (a1);
\draw[thick] (a1) to [out=90,in=90]  (a2);
\draw[thick] (a3) to [out=-90,in=-90]  (a4);
\draw[thick,->] (a4) to [out=90,in=-90]  (a5);
\draw[line width=5, color=white] (a2) to [out=-90,in=90]  (a3);
\draw[thick] (a2) to [out=-90,in=90]  (a3);
\end{tikzpicture},
\quad
 \xi^-:=
\begin{tikzpicture}[baseline=15,scale=.5]
 \coordinate (a0) at (0,0);
 \coordinate (a1) at (1,2);
 \coordinate (a2) at (0,2);
 \coordinate (a3) at (2,.5);
\coordinate (a4) at (1,.5);
\coordinate (a5) at (2,2.5);
\draw[thick] (a2) to [out=-90,in=90]  (a3);

\draw[thick] (a1) to [out=90,in=90]  (a2);
\draw[thick] (a3) to [out=-90,in=-90]  (a4);

\draw[line width=5, color=white] (a0) to [out=90,in=-90]  (a1);
\draw[thick] (a0) to [out=90,in=-90]  (a1);

\draw[line width=5, color=white] (a4) to [out=90,in=-90]  (a5);
\draw[thick,->] (a4) to [out=90,in=-90]  (a5);
\end{tikzpicture},
\quad
\xi^n:=
\left\{
\begin{array}{cl}
 \begin{tikzpicture}[baseline]
 \draw[thick,->] (0,0)--(0,.3);
\end{tikzpicture} 
& \text{if }  n=0;      \\
  \underbrace{\xi^{\operatorname{sgn}(n)}\circ \dots\circ \xi^{\operatorname{sgn}(n)}}_{|n|\text{ times}}&\text{if }  n\ne 0    
\end{array}
\right.
\end{equation}
where $n\in\Z$, $\operatorname{sgn}(n):=n/|n|$ and we identify the signs $\pm$ with the numbers $\pm1$.

\begin{remark}
 Any normal long knot diagram has an even number of crossings. In particular, we have $w(\xi^n)=2n$.
\end{remark}

\section{Invariants of long knots from rigid $R$-matrices}\label{sec:ilkrrm}
We say  an object $G$ of a monoidal category $\mathcal{C}$ (with tensor product $\otimes$ and unit object $\mathbb{I}$) admits a left \emph{\color{blue} adjoint} if there exists an object $F$ and morphisms 
\begin{equation}
\varepsilon\colon F\otimes G\to \mathbb{I},\quad \eta\colon \mathbb{I}\to G\otimes F
\end{equation}
 such that
\begin{equation}
(\varepsilon\otimes \operatorname{id}_{F})\circ( \operatorname{id}_{F}\otimes \eta)=\operatorname{id}_{F},\quad ( \operatorname{id}_{G}\otimes \varepsilon)\circ(\eta\otimes \operatorname{id}_{G})=\operatorname{id}_{G}.
\end{equation}
In that case, the quadruple $(F,G,\varepsilon,\eta)$ is called  \emph{\color{blue} adjunction} in  $\mathcal{C}$.
\begin{definition}
 Let $\mathcal{C}$ be a monoidal category. An  \emph{\color{blue} R-matrix} over  an object $G\in\operatorname{Ob}\mathcal{C}$ is an element $r\in\operatorname{Aut}(G\otimes G)$ that satisfies the Yang--Baxter relation
 \begin{equation}
(r\otimes \operatorname{id}_G)\circ(\operatorname{id}_G\otimes r )\circ(r\otimes \operatorname{id}_G)=(\operatorname{id}_G\otimes r )\circ(r\otimes \operatorname{id}_G)\circ(\operatorname{id}_G\otimes r ).
\end{equation}
\end{definition}
\begin{definition}
Let $(F,G,\varepsilon,\eta)$  be an adjunction in  a monoidal category $\mathcal{C}$. An $R$-matrix $r$ over $G$ is called \emph{\color{blue} rigid} if the morphisms 
\begin{equation}\label{eq:tild-rpm1}
\widetilde{r^{\pm1}}:=(\varepsilon\otimes \operatorname{id}_{G\otimes F})\circ(\operatorname{id}_{F}\otimes r^{\pm1}\otimes\operatorname{id}_{ F})\circ( \operatorname{id}_{F\otimes G}\otimes \eta)
\end{equation}
are invertible.
\end{definition}
We also denote
\begin{equation}\label{eq:tild-tild-rpm1}
\widetilde{\widetilde{r^{\pm1}}}:=(\varepsilon\otimes \operatorname{id}_{F\otimes F})\circ(\operatorname{id}_{F}\otimes \widetilde{r^{\pm1}}\otimes\operatorname{id}_{ F})\circ( \operatorname{id}_{F\otimes F}\otimes \eta).
\end{equation}
One easily checks the identity
\begin{equation}
\widetilde{\widetilde{r^{-1}}}=\left(\widetilde{\widetilde{r}}\right)^{-1}.
\end{equation}

Associated to a rigid R-matrix $r$ over $G$ with an adjunction $(F,G,\varepsilon,\eta)$, the \emph{\color{blue} Reshetikhin--Turaev functor} $RT_r$ associates to any normal long knot diagram $D$ the endomorphism $RT_r(D)\colon G\to G$ obtained as follows.

Assuming that the non-trivial part of $D$ is contained in $\R\times [0,1]$, there exists a finite sequence of real numbers $0=t_0<t_1<\dots<t_{n-1}<t_n=1$ such that, for any $i\in \{0,\dots,n-1\}$,  the intersection $D_i:=D\cap(\R\times [t_i,t_{i+1}])$ is an ordered (from left to right) finite sequence of connected components each of which is isotopic relative to boundary either to one of the four types of segments 
\begin{equation}
\begin{tikzpicture}[yscale=.5,baseline]
\draw[thick,->] (0,0) to [out=90,in=-90] (0,1);
\end{tikzpicture}\ ,
\quad
\begin{tikzpicture}[yscale=.5,baseline]
\draw[thick,<-] (0,0) to [out=90,in=-90] (0,1);
\end{tikzpicture}\ ,
\quad
\begin{tikzpicture}[xscale=.5,baseline]
\draw[thick,<-] (0,0) to [out=90,in=90] (1,0);
\end{tikzpicture}\ ,
\quad
\begin{tikzpicture}[xscale=.5,baseline=15]
\draw[thick,<-] (0,1) to [out=-90,in=-90] (1,1);
\end{tikzpicture}
\end{equation}
or to one of the eight types of crossings
\begin{equation}
\begin{tikzpicture}[scale=.5,baseline]
\draw[thick,<-] (0,1) to [out=-90,in=90] (1,0);
\draw[line width=3pt,white] (1,1) to [out=-90,in=90] (0,0);
\draw[thick,<-] (1,1) to [out=-90,in=90] (0,0);
\end{tikzpicture}\ ,
\quad
\begin{tikzpicture}[scale=.5,baseline]
\draw[thick,<-] (1,1) to [out=-90,in=90] (0,0);
\draw[line width=3pt,white] (0,1) to [out=-90,in=90] (1,0);
\draw[thick,<-] (0,1) to [out=-90,in=90] (1,0);
\end{tikzpicture}\ ,
\quad
\begin{tikzpicture}[scale=.5,baseline]
\draw[thick,->] (1,1) to [out=-90,in=90] (0,0);
\draw[line width=3pt,white] (0,1) to [out=-90,in=90] (1,0);
\draw[thick,<-] (0,1) to [out=-90,in=90] (1,0);
\end{tikzpicture}\ ,
\quad
\begin{tikzpicture}[scale=.5,baseline]
\draw[thick,<-] (0,1) to [out=-90,in=90] (1,0);
\draw[line width=3pt,white] (1,1) to [out=-90,in=90] (0,0);
\draw[thick,->] (1,1) to [out=-90,in=90] (0,0);
\end{tikzpicture}\ ,
\quad
\begin{tikzpicture}[scale=.5,baseline]
\draw[thick,->] (0,1) to [out=-90,in=90] (1,0);
\draw[line width=3pt,white] (1,1) to [out=-90,in=90] (0,0);
\draw[thick,->] (1,1) to [out=-90,in=90] (0,0);
\end{tikzpicture}\ ,
\quad
\begin{tikzpicture}[scale=.5,baseline]
\draw[thick,->] (1,1) to [out=-90,in=90] (0,0);
\draw[line width=3pt,white] (0,1) to [out=-90,in=90] (1,0);
\draw[thick,->] (0,1) to [out=-90,in=90] (1,0);
\end{tikzpicture}\ ,
\quad
\begin{tikzpicture}[scale=.5,baseline]
\draw[thick,<-] (1,1) to [out=-90,in=90] (0,0);
\draw[line width=3pt,white] (0,1) to [out=-90,in=90] (1,0);
\draw[thick,->] (0,1) to [out=-90,in=90] (1,0);
\end{tikzpicture}\ ,
\quad
\begin{tikzpicture}[scale=.5,baseline]
\draw[thick,->] (0,1) to [out=-90,in=90] (1,0);
\draw[line width=3pt,white] (1,1) to [out=-90,in=90] (0,0);
\draw[thick,<-] (1,1) to [out=-90,in=90] (0,0);
\end{tikzpicture}\ .
\end{equation}
To such an intersection,  we associate 
 a morphism $f_i$ in $\mathcal{C}$ by taking the tensor product (from left to right) of the morphisms associated to the connected fragments of $D_i$ according to the following rules:
\begin{equation}
\begin{tikzpicture}[yscale=.5,baseline=3]
\draw[thick,->] (0,0) to [out=90,in=-90] (0,1);
\end{tikzpicture}\ \mapsto \operatorname{id}_G,
\quad
\begin{tikzpicture}[yscale=.5,baseline=4]
\draw[thick,<-] (0,0) to [out=90,in=-90] (0,1);
\end{tikzpicture}\ \mapsto \operatorname{id}_{F},
\quad
\begin{tikzpicture}[xscale=.5,baseline=2]
\draw[thick,<-] (0,0) to [out=90,in=90] (1,0);
\end{tikzpicture}\ \mapsto \varepsilon,
\quad
\begin{tikzpicture}[xscale=.5,baseline=22]
\draw[thick,<-] (0,1) to [out=-90,in=-90] (1,1);
\end{tikzpicture}\ \mapsto\eta,
\end{equation}
\begin{equation}
\begin{tikzpicture}[scale=.5,baseline=3]
\draw[thick,<-] (0,1) to [out=-90,in=90] (1,0);
\draw[line width=3pt,white] (1,1) to [out=-90,in=90] (0,0);
\draw[thick,<-] (1,1) to [out=-90,in=90] (0,0);
\end{tikzpicture}\ \mapsto r,
\quad
\begin{tikzpicture}[scale=.5,baseline=3]
\draw[thick,<-] (1,1) to [out=-90,in=90] (0,0);
\draw[line width=3pt,white] (0,1) to [out=-90,in=90] (1,0);
\draw[thick,<-] (0,1) to [out=-90,in=90] (1,0);
\end{tikzpicture}\ \mapsto r^{-1},
\end{equation}
\begin{equation}
\begin{tikzpicture}[scale=.5,baseline=3]
\draw[thick,->] (1,1) to [out=-90,in=90] (0,0);
\draw[line width=3pt,white] (0,1) to [out=-90,in=90] (1,0);
\draw[thick,<-] (0,1) to [out=-90,in=90] (1,0);
\end{tikzpicture}\ \mapsto \widetilde{r},
\quad
\begin{tikzpicture}[scale=.5,baseline=3]
\draw[thick,<-] (0,1) to [out=-90,in=90] (1,0);
\draw[line width=3pt,white] (1,1) to [out=-90,in=90] (0,0);
\draw[thick,->] (1,1) to [out=-90,in=90] (0,0);
\end{tikzpicture}\ \mapsto  \widetilde{r^{-1}},
\end{equation}
\begin{equation}
\begin{tikzpicture}[scale=.5,baseline=3]
\draw[thick,->] (0,1) to [out=-90,in=90] (1,0);
\draw[line width=3pt,white] (1,1) to [out=-90,in=90] (0,0);
\draw[thick,->] (1,1) to [out=-90,in=90] (0,0);
\end{tikzpicture}\ \mapsto \widetilde{ \widetilde{r}},
\quad
\begin{tikzpicture}[scale=.5,baseline=3]
\draw[thick,->] (1,1) to [out=-90,in=90] (0,0);
\draw[line width=3pt,white] (0,1) to [out=-90,in=90] (1,0);
\draw[thick,->] (0,1) to [out=-90,in=90] (1,0);
\end{tikzpicture}\ \mapsto  \widetilde{ \widetilde{r^{-1}}},
\end{equation}
\begin{equation}
\begin{tikzpicture}[scale=.5,baseline=3]
\draw[thick,<-] (1,1) to [out=-90,in=90] (0,0);
\draw[line width=3pt,white] (0,1) to [out=-90,in=90] (1,0);
\draw[thick,->] (0,1) to [out=-90,in=90] (1,0);
\end{tikzpicture}\ \mapsto \left( \widetilde{r^{-1}}\right)^{-1},
\quad
\begin{tikzpicture}[scale=.5,baseline=3]
\draw[thick,->] (0,1) to [out=-90,in=90] (1,0);
\draw[line width=3pt,white] (1,1) to [out=-90,in=90] (0,0);
\draw[thick,<-] (1,1) to [out=-90,in=90] (0,0);
\end{tikzpicture}\ \mapsto ( \widetilde{r})^{-1}.
\end{equation}
The morphism $RT_r(D)\colon G\to G$ associated to $D$ is obtained as the composition
\begin{equation}
RT_r(D):=f_{n-1}\circ \dots\circ f_1\circ f_0.
\end{equation}
\begin{example}
 Let $D=\xi^-$ defined in \eqref{eq:xi+-}. Then $RT_r(D)=f_3\circ f_2\circ f_1\circ f_0$ where
 \begin{equation}
f_0= \operatorname{id}_G\otimes \eta,\quad f_1= \operatorname{id}_G\otimes\left( \widetilde{r}\right)^{-1},\quad f_2=\left( \widetilde{r}\right)^{-1} \otimes\operatorname{id}_G,\quad f_3=\varepsilon \otimes\operatorname{id}_G.
\end{equation}

\end{example}
\begin{theorem}[\cite{MR1036112,MR1025161,MR2796628}]
 Let $r$ be a rigid $R$-matrix over an object $G$ of a monoidal category $\mathcal{C}$ with an adjunction $(F,G,\varepsilon,\eta)$. Then, for any long knot diagram $D$, the element 
 \begin{equation}
J_r(D):=RT_r(\dot{D}\circ \xi^{-w(\dot{D})/2})\in \operatorname{End}(G)
\end{equation}
depends on only the Reidemeister equivalence class of $D$.
\end{theorem}
\begin{proof}
 Let $\dot{\sim}$ be the equivalence relation on the set of normal long knot diagrams generated by the oriented versions of the Reidemeister moves RII and RIII and the moves R$0^\pm$ defined by the pictures
 \begin{equation}
\begin{tikzpicture}[xscale=.5,yscale=1.5,baseline=5]
\draw[thick,->] (1,0) to [out=90,in=90] (0,0);
\draw[thick] (.5,0)--(0,.4);
\end{tikzpicture}
\ \stackrel{\text{ R}0^+}{\longleftrightarrow}\ 
\begin{tikzpicture}[xscale=.5,yscale=1.5,baseline=5]
\draw[thick,->] (1,0) to [out=90,in=90] (0,0);
\draw[thick] (.5,0)--(1,.4);
\end{tikzpicture}
\ ,\qquad
\begin{tikzpicture}[xscale=.5,yscale=1.5,baseline=-10]
\draw[thick,->] (1,0) to [out=-90,in=-90] (0,0);
\draw[thick] (.5,0)--(0,-.4);
\end{tikzpicture}
\ \stackrel{\text{ R}0^-}{\longleftrightarrow}\ 
\begin{tikzpicture}[xscale=.5,yscale=1.5,baseline=-10]
\draw[thick,->] (1,0) to [out=-90,in=-90] (0,0);
\draw[thick] (.5,0)--(1,-.4);
\end{tikzpicture}
\end{equation}
with two possible orientations for the straight segment and two possibilities for the crossing. The strategy of the proof is to show first the implication
\begin{equation}\label{eq:dot-sim-=}
\dot{D}\ \dot{\sim}\  \dot{D}'\Rightarrow RT_r(\dot{D})=RT_r(\dot{D}')
\end{equation}
which, by taking into account the implication
\begin{equation}\label{eq:dot-sim-=writhe}
\dot{D}\ \dot{\sim}\  \dot{D}'\Rightarrow w(\dot{D})=w(\dot{D}'),
\end{equation}
 ensures the invariance of $J_r(D)$ under all Reidemeister moves RII and RIII, and then to verify the invariance under the Reidemeister moves RI. It is in this last part of the proof where the correction of $\dot{D}$ by $\xi^{-w(\dot{D})/2}$ is crucial.

Invariance of Reshetikhin--Turaev functor with respect to moves   $\text{ R}0^\pm$ follows from the equivalences

 \begin{equation}
\begin{tikzpicture}[xscale=.5,yscale=1.5,baseline=5]
\draw[thick,->] (1,0) to [out=90,in=90] (0,0);
\draw[thick] (.5,0)--(0,.4);
\end{tikzpicture}
\ \dot{\sim}\ 
\begin{tikzpicture}[xscale=.5,yscale=1.5,baseline=5]
\draw[thick,->] (1,0) to [out=90,in=90] (0,0);
\draw[thick] (.5,0)--(1,.4);
\end{tikzpicture}
\Leftrightarrow
\begin{tikzpicture}[xscale=.5,yscale=1.5,baseline=5]
\draw[thick,->] (1,0) to [out=90,in=90] (0,0);
\draw[thick] (2,.4) to [out=-90,in=-90] (1,0);
\draw[thick] (.5,0)--(0,.4);
\end{tikzpicture}
\ \dot{\sim}\ 
\begin{tikzpicture}[xscale=.5,yscale=1.5,baseline=5]
\draw[thick,->] (1,0) to [out=90,in=90] (0,0);
\draw[thick] (2,.4) to [out=-90,in=-90] (1,0);
\draw[thick] (.5,0)--(1,.4);
\end{tikzpicture}
\Leftrightarrow
\begin{tikzpicture}[scale=.5,baseline=3]
\draw[thick,->] (1,1) to [out=-90,in=90] (0,0);
\draw[thick] (0,1) to [out=-90,in=90] (1,0);
\end{tikzpicture}
\ \dot{\sim}\ 
\begin{tikzpicture}[xscale=.5,yscale=1.5,baseline=5]
\draw[thick,->] (1,0) to [out=90,in=90] (0,0);
\draw[thick] (2,.4) to [out=-90,in=-90] (1,0);
\draw[thick] (.5,0)--(1,.4);
\end{tikzpicture},
\end{equation}
\begin{equation}
\begin{tikzpicture}[xscale=.5,yscale=1.5,baseline=-10]
\draw[thick,->] (1,0) to [out=-90,in=-90] (0,0);
\draw[thick] (.5,0)--(0,-.4);
\end{tikzpicture}
\ \dot{\sim}\ 
\begin{tikzpicture}[xscale=.5,yscale=1.5,baseline=-10]
\draw[thick,->] (1,0) to [out=-90,in=-90] (0,0);
\draw[thick] (.5,0)--(1,-.4);
\end{tikzpicture}
\Leftrightarrow
\begin{tikzpicture}[xscale=.5,yscale=1.5,baseline=-10]
\draw[thick] (1,0) to [out=-90,in=-90] (0,0);
\draw[thick,->] (0,0) to [out=90,in=90] (-1,-.4);
\draw[thick] (.5,0)--(0,-.4);
\end{tikzpicture}
\ \dot{\sim}\ 
\begin{tikzpicture}[xscale=.5,yscale=1.5,baseline=-10]
\draw[thick] (1,0) to [out=-90,in=-90] (0,0);
\draw[thick,->] (0,0) to [out=90,in=90] (-1,-.4);
\draw[thick] (.5,0)--(1,-.4);
\end{tikzpicture}
\Leftrightarrow
\begin{tikzpicture}[xscale=.5,yscale=1.5,baseline=-10]
\draw[thick] (1,0) to [out=-90,in=-90] (0,0);
\draw[thick,->] (0,0) to [out=90,in=90] (-1,-.4);
\draw[thick] (.5,0)--(0,-.4);
\end{tikzpicture}
\ \dot{\sim}\ 
\begin{tikzpicture}[scale=.5,baseline=3]
\draw[thick,->] (1,1) to [out=-90,in=90] (0,0);
\draw[thick] (0,1) to [out=-90,in=90] (1,0);
\end{tikzpicture}
\end{equation}
and the definitions \eqref{eq:tild-rpm1} and  \eqref{eq:tild-tild-rpm1} of $\widetilde{r^{\pm1}}$ and  $\widetilde{\widetilde{r^{\pm1}}}$.

Invariance of $RT_r$ with respect to oriented $\text{RII}$ moves are easily checked first for eight basic moves
\begin{equation}
\begin{tikzpicture}[xscale=0.5,yscale=1,baseline=10]
\draw[thick] (1,0) to [out=135,in=-90] (.3,.5);
\draw[thick,->] (.3,.5) to [out=90,in=-135] (1,1);
\draw[line width=3pt,white] (0,0) to [out=45,in=-90] (.7,.5);
\draw[line width=3pt,white] (.7,.5) to [out=90,in=-45] (0,1);
\draw[thick] (0,0) to [out=45,in=-90] (.7,.5);
\draw[thick,->] (.7,.5) to [out=90,in=-45] (0,1);
\end{tikzpicture}
\ \dot{\sim}\ 
\begin{tikzpicture}[xscale=0.5,yscale=1,baseline=10]
\draw[thick,->] (1,0) to [out=135,in=-135] (1,1);
\draw[thick,->] (0,0) to [out=45,in=-45] (0,1);
\end{tikzpicture}
\ \dot{\sim}\ 
\begin{tikzpicture}[xscale=0.5,yscale=1,baseline=10]
\draw[thick] (0,0) to [out=45,in=-90] (.7,.5);
\draw[thick,->] (.7,.5) to [out=90,in=-45] (0,1);
\draw[line width=3pt,white] (1,0) to [out=135,in=-90] (.3,.5);
\draw[line width=3pt,white] (.3,.5) to [out=90,in=-135] (1,1);
\draw[thick] (1,0) to [out=135,in=-90] (.3,.5);
\draw[thick,->] (.3,.5) to [out=90,in=-135] (1,1);
\end{tikzpicture}
\stackrel{RT_r}{\longmapsto}  r^{-1}\circ r=\operatorname{id}_{G\otimes G}=r\circ r^{-1},
\end{equation}
\begin{equation}
\begin{tikzpicture}[xscale=0.5,yscale=1,baseline=10]
\draw[thick,->]  (.3,.5) to [out=-90,in=135](1,0);
\draw[thick] (.3,.5) to [out=90,in=-135] (1,1);
\draw[line width=3pt,white] (0,0) to [out=45,in=-90] (.7,.5);
\draw[line width=3pt,white] (.7,.5) to [out=90,in=-45] (0,1);
\draw[thick] (0,0) to [out=45,in=-90] (.7,.5);
\draw[thick,->] (.7,.5) to [out=90,in=-45] (0,1);
\end{tikzpicture}
\ \dot{\sim}\ 
\begin{tikzpicture}[xscale=0.5,yscale=1,baseline=10]
\draw[thick,->]  (1,1) to [out=-135,in=135] (1,0);
\draw[thick,->] (0,0) to [out=45,in=-45] (0,1);
\end{tikzpicture}
\ \dot{\sim}\ 
\begin{tikzpicture}[xscale=0.5,yscale=1,baseline=10]
\draw[thick] (0,0) to [out=45,in=-90] (.7,.5);
\draw[thick,->] (.7,.5) to [out=90,in=-45] (0,1);
\draw[line width=3pt,white] (1,0) to [out=135,in=-90] (.3,.5);
\draw[line width=3pt,white] (.3,.5) to [out=90,in=-135] (1,1);
\draw[thick,<-](1,0) to [out=135,in=-90]  (.3,.5);
\draw[thick] (.3,.5) to [out=90,in=-135] (1,1);
\end{tikzpicture}
\stackrel{RT_r}{\longmapsto}    \widetilde{r}\circ  (\widetilde{r})^{-1}=\operatorname{id}_{G\otimes F}= \widetilde{r^{-1}}\circ  \left(\widetilde{r^{-1}}\right)^{-1},
\end{equation}
\begin{equation}
\begin{tikzpicture}[xscale=0.5,yscale=1,baseline=10]
\draw[thick] (1,0) to [out=135,in=-90] (.3,.5);
\draw[thick,->] (.3,.5) to [out=90,in=-135] (1,1);
\draw[line width=3pt,white] (0,0) to [out=45,in=-90] (.7,.5);
\draw[line width=3pt,white] (.7,.5) to [out=90,in=-45] (0,1);
\draw[thick,<-] (0,0) to [out=45,in=-90] (.7,.5);
\draw[thick] (.7,.5) to [out=90,in=-45] (0,1);
\end{tikzpicture}
\ \dot{\sim}\ 
\begin{tikzpicture}[xscale=0.5,yscale=1,baseline=10]
\draw[thick,->] (1,0) to [out=135,in=-135] (1,1);
\draw[thick,<-] (0,0) to [out=45,in=-45] (0,1);
\end{tikzpicture}
\ \dot{\sim}\ 
\begin{tikzpicture}[xscale=0.5,yscale=1,baseline=10]
\draw[thick,<-] (0,0) to [out=45,in=-90] (.7,.5);
\draw[thick] (.7,.5) to [out=90,in=-45] (0,1);
\draw[line width=3pt,white] (1,0) to [out=135,in=-90] (.3,.5);
\draw[line width=3pt,white] (.3,.5) to [out=90,in=-135] (1,1);
\draw[thick] (1,0) to [out=135,in=-90] (.3,.5);
\draw[thick,->] (.3,.5) to [out=90,in=-135] (1,1);
\end{tikzpicture}
\stackrel{RT_r}{\longmapsto}  
\left(\widetilde{r^{-1}}\right)^{-1}\circ \widetilde{r^{-1}}=\operatorname{id}_{F\otimes G}= (\widetilde{r})^{-1}\circ \widetilde{r},
\end{equation}
\begin{equation}
\begin{tikzpicture}[xscale=0.5,yscale=1,baseline=10]
\draw[thick,<-] (1,0) to [out=135,in=-90] (.3,.5);
\draw[thick] (.3,.5) to [out=90,in=-135] (1,1);
\draw[line width=3pt,white] (0,0) to [out=45,in=-90] (.7,.5);
\draw[line width=3pt,white] (.7,.5) to [out=90,in=-45] (0,1);
\draw[thick,<-] (0,0) to [out=45,in=-90] (.7,.5);
\draw[thick] (.7,.5) to [out=90,in=-45] (0,1);
\end{tikzpicture}
\ \dot{\sim}\ 
\begin{tikzpicture}[xscale=0.5,yscale=1,baseline=10]
\draw[thick,<-] (1,0) to [out=135,in=-135] (1,1);
\draw[thick,<-] (0,0) to [out=45,in=-45] (0,1);
\end{tikzpicture}
\ \dot{\sim}\ 
\begin{tikzpicture}[xscale=0.5,yscale=1,baseline=10]
\draw[thick,<-] (0,0) to [out=45,in=-90] (.7,.5);
\draw[thick] (.7,.5) to [out=90,in=-45] (0,1);
\draw[line width=3pt,white] (1,0) to [out=135,in=-90] (.3,.5);
\draw[line width=3pt,white] (.3,.5) to [out=90,in=-135] (1,1);
\draw[thick,<-] (1,0) to [out=135,in=-90] (.3,.5);
\draw[thick] (.3,.5) to [out=90,in=-135] (1,1);
\end{tikzpicture}
\stackrel{RT_r}{\longmapsto}  
\widetilde{\widetilde{r^{-1}}}\circ \widetilde{\widetilde{r}}=\operatorname{id}_{F\otimes F}=  \widetilde{\widetilde{r}}\circ\widetilde{\widetilde{r^{-1}}},
\end{equation}
and then for two composite moves
\begin{equation}
\begin{tikzpicture}[xscale=1,yscale=.5,baseline=5]
\draw[thick] (1,0) to [out=90,in=0] (.5,.9);
\draw[thick,->] (.5,.9) to [out=180,in=90] (0,0);
\draw[thick] (1,1) to [out=-90,in=0] (.5,.1);
\draw[thick,->] (.5,.1) to [out=180,in=-90] (0,1);
\end{tikzpicture}
\ \dot{\sim}\ 
\begin{tikzpicture}[xscale=1,yscale=.5,baseline=5]
\draw[thick] (.2,0) to [out=45,in=-135] (.6,.4);
\draw[thick] (.6,.4) to [out=45,in=0] (.5,.9);
\draw[thick,->] (.5,.9) to [out=180,in=90] (0,0);
\draw[thick] (1,1) to [out=-90,in=0] (.5,.1);
\draw[thick,->] (.5,.1) to [out=180,in=-90] (0,1);
\end{tikzpicture}
\ \dot{\sim}\ 
\begin{tikzpicture}[xscale=1,yscale=.5,baseline=5]
\draw[thick] (.2,0) to [out=45,in=-135] (.6,.4);
\draw[thick] (.6,.4) to [out=45,in=0] (.5,.9);
\draw[thick,->] (.5,.9) to [out=180,in=90] (0,0);
\draw[thick] (1,1) to [out=-90,in=0] (.5,.1);
\draw[thick] (.5,.1) to [out=180,in=-135] (.4,.6);
\draw[thick,->] (.4,.6) to [out=45,in=-135] (.8,1);
\end{tikzpicture}
\ \dot{\sim}\ 
\begin{tikzpicture}[xscale=1,yscale=.5,baseline=5]
\draw[thick] (.2,0) to [out=90,in=0] (.3,.9);
\draw[thick,->] (.3,.9) to [out=180,in=90] (0,0);

\draw[thick] (1,1) to [out=-90,in=0] (.7,.1);
\draw[thick,->] (.7,.1) to [out=180,in=-90] (.8,1);
\end{tikzpicture}
\ = \ 
\begin{tikzpicture}[xscale=1,yscale=.5,baseline=5]
\draw[thick,->] (1,0) to [out=135,in=45] (0,0);
\draw[thick,->] (1,1) to [out=-135,in=-45] (0,1);
\end{tikzpicture}
\end{equation}
with two possible choices for the crossings.

In order to check invariance of $RT_r$ with respect to RIII moves, we remark that altogether there are 48 such moves which can indexed by the set $\operatorname{Sym}(3)\times \{\pm1\}^3$ as follows.

Given an RIII move, we enumerate the strands that intervene the move by following their bottom open ends from left to right
\begin{equation}
\begin{tikzpicture}[xscale=.5,yscale=1,baseline=10]
\draw[thick] (0,0)--(2,1);
\draw[thick] (1,0) to [out=90,in=-90] (.5,.5);
\draw[thick] (1,1) to [out=-90,in=90] (.5,.5);
\draw[thick] (2,0)-- (0,1);
\end{tikzpicture}
\ \dot{\sim}\ 
\begin{tikzpicture}[xscale=.5,yscale=1,baseline=10]
\draw[thick] (0,0)--(2,1);
\draw[thick] (1,0) to [out=90,in=-90] (1.5,.5);
\draw[thick] (1,1) to [out=-90,in=90] (1.5,.5);
\draw[thick] (2,0)-- (0,1);
\end{tikzpicture}
\quad \rightsquigarrow\quad
\begin{tikzpicture}[xscale=.5,yscale=1,baseline=10]
\node (a) at (0,0){\tiny1};
\node (b) at (1,0){\tiny2};
\node (c) at (2,0){\tiny3};
\draw[thick] (a)--(2,1);
\draw[thick] (b) to [out=90,in=-90] (.5,.5);
\draw[thick] (1,1) to [out=-90,in=90] (.5,.5);
\draw[thick] (c)-- (0,1);
\end{tikzpicture}
\ \dot{\sim}\ 
\begin{tikzpicture}[xscale=.5,yscale=1,baseline=10]
\node (a) at (0,0){\tiny1};
\node (b) at (1,0){\tiny2};
\node (c) at (2,0){\tiny3};
\draw[thick] (a)--(2,1);
\draw[thick] (b) to [out=90,in=-90] (1.5,.5);
\draw[thick] (1,1) to [out=-90,in=90] (1.5,.5);
\draw[thick] (c)-- (0,1);
\end{tikzpicture}
\end{equation}
and we define the associated element $(\sigma,\varepsilon)\in\operatorname{Sym}(3)\times \{\pm1\}^3$ by the conditions that for any $i\in \{1,2,3\}$, $\sigma(i)$ is the number of arcs on the $i$-th strand and
$\varepsilon_i=1$ if $i$-th strand is oriented upwards. For example, the RIII move
\begin{equation}
\begin{tikzpicture}[xscale=1,yscale=1,baseline=10]

\draw[thick,<-] (1,0) to [out=90,in=-90] (.5,.5);
\draw[thick] (1,1) to [out=-90,in=90] (.5,.5);

\draw[line width=3pt,white]  (2,0)-- (0,1);
\draw[thick,->] (2,0)-- (0,1);
\draw[line width=3pt,white] (0,0)--(2,1);
\draw[thick,->] (0,0)--(2,1);
\end{tikzpicture}
\ \dot{\sim}\ 
\begin{tikzpicture}[xscale=1,yscale=1,baseline=10]

\draw[thick,<-] (1,0) to [out=90,in=-90] (1.5,.5);
\draw[thick] (1,1) to [out=-90,in=90] (1.5,.5);

\draw[line width=3pt,white]  (2,0)-- (0,1);
\draw[thick,->] (2,0)-- (0,1);
\draw[line width=3pt,white] (0,0)--(2,1);
\draw[thick,->] (0,0)--(2,1);
\end{tikzpicture}
\end{equation}
corresponds to permutation $\sigma=(2,3)=(1)(2,3)$ and $\varepsilon=(1,-1,1)$ while the pair $(\sigma=\operatorname{id}, \varepsilon=(1,1,1))$ 
corresponds to the reference move associated to the Yang--Baxter relation
\begin{equation}\label{eq.basriii}
\begin{tikzpicture}[xscale=1,yscale=1,baseline=10]
\draw[thick,->] (2,0)-- (0,1);
\draw[line width=3pt,white]  (1,1) to [out=-90,in=90]  (0.5,.5);
\draw[thick] (1,0) to [out=90,in=-90] (0.5,.5);
\draw[thick,<-] (1,1) to [out=-90,in=90] (0.5,.5);

\draw[line width=3pt,white] (0,0)--(2,1);
\draw[thick,->] (0,0)--(2,1);
\end{tikzpicture}
\ \dot{\sim}\ 
\begin{tikzpicture}[xscale=1,yscale=1,baseline=10]
\draw[thick,->] (2,0)-- (0,1);
\draw[line width=3pt,white] (1,0) to [out=90,in=-90] (1.5,.5);
\draw[thick] (1,0) to [out=90,in=-90] (1.5,.5);
\draw[thick,<-] (1,1) to [out=-90,in=90] (1.5,.5);

\draw[line width=3pt,white] (0,0)--(2,1);
\draw[thick,->] (0,0)--(2,1);
\end{tikzpicture}
\ \stackrel{RT_r}{\longmapsto}\   r_1\circ r_2\circ r_1=r_2\circ r_1\circ r_2
\end{equation}
with the notations $r_1:=r\otimes \operatorname{id}_G$ and $r_2:=\operatorname{id}_G\otimes r $. 

One can now show that by using the moves RII and $\text{R}0^\pm$, any RIII move is equivalent to the reference move~\eqref{eq.basriii}.

Indeed, in the case $\varepsilon_1=-1$, we have the equivalences
\begin{equation}
\begin{tikzpicture}[xscale=.5,yscale=1,baseline=10]
\draw[thick,<-] (0,0)--(2,1);
\draw[thick] (1,0) to [out=90,in=-90] (1.5,.5);
\draw[thick] (1,1) to [out=-90,in=90] (1.5,.5);
\draw[thick] (2,0)-- (0,1);
\end{tikzpicture}
\ \dot{\sim}\ 
\begin{tikzpicture}[xscale=.5,yscale=1,baseline=10]
\draw[thick,<-] (0,0)--(2,1);
\draw[thick] (1,0) to [out=90,in=-90] (.5,.5);
\draw[thick] (1,1) to [out=-90,in=90] (.5,.5);
\draw[thick] (2,0)-- (0,1);
\end{tikzpicture}
\Leftrightarrow
\begin{tikzpicture}[xscale=.5,yscale=1,baseline=10]
\draw[thick,<-] (-1,1) to [out=-45,in=-135] (1,.4);
\draw[thick] (3,0) to [out=135,in=45] (1,.4);
\draw[thick] (1,0) to [out=90,in=-90] (1.5,.5);
\draw[thick] (1,1) to [out=-90,in=90] (1.5,.5);
\draw[thick] (2,0)-- (0,1);
\end{tikzpicture}
\ \dot{\sim}\ 
\begin{tikzpicture}[xscale=.5,yscale=1,baseline=10]
\draw[thick,<-] (-1,1) to [out=-45,in=-135] (1,.4);
\draw[thick] (3,0) to [out=135,in=45] (1,.4);
\draw[thick] (1,0) to [out=90,in=-90] (.5,.5);
\draw[thick] (1,1) to [out=-90,in=90] (.5,.5);
\draw[thick] (2,0)-- (0,1);
\end{tikzpicture}
\Leftrightarrow
\begin{tikzpicture}[xscale=.5,yscale=1,baseline=10]
\draw[thick] (0,0)--(2,1);
\draw[thick] (1,0) to [out=90,in=-90] (.5,.5);
\draw[thick] (1,1) to [out=-90,in=90] (.5,.5);
\draw[thick,->] (2,0)-- (0,1);
\end{tikzpicture}
\ \dot{\sim}\ 
\begin{tikzpicture}[xscale=.5,yscale=1,baseline=10]
\draw[thick] (0,0)--(2,1);
\draw[thick] (1,0) to [out=90,in=-90] (1.5,.5);
\draw[thick] (1,1) to [out=-90,in=90] (1.5,.5);
\draw[thick,->] (2,0)-- (0,1);
\end{tikzpicture}
\end{equation}
which imply the equivalence
\begin{equation}\label{eq:equiv(-1,.)}
(\sigma,(-1,\varepsilon_2,\varepsilon_3))\Leftrightarrow (\sigma\circ(1,2,3),(\varepsilon_2,\varepsilon_3,1))
\end{equation}
in the set $\operatorname{Sym}(3)\times \{\pm1\}^3$ thus allowing to reduce the number of negative components of $\varepsilon$. Additionally, a right action of the permutation group $\operatorname{Sym}(3)$ on the set $\operatorname{Sym}(3)\times \{\pm1\}^3$ is induced by the equivalences
\begin{equation}
\begin{tikzpicture}[xscale=.5,yscale=1,baseline=10]
\draw[thick] (0,0) -- (2,1);
\draw[thick] (1,0) to [out=90,in=-90] (1.5,.5);
\draw[thick] (1,1) to [out=-90,in=90] (1.5,.5);
\draw[thick] (2,0)--(0,1);
\end{tikzpicture}
\ \dot{\sim}\ 
\begin{tikzpicture}[xscale=.5,yscale=1,baseline=10]
\draw[thick] (0,0)--(2,1);
\draw[thick] (1,0) to [out=90,in=-90] (.5,.5);
\draw[thick] (1,1) to [out=-90,in=90] (.5,.5);
\draw[thick] (2,0)--(0,1);
\end{tikzpicture}
\Leftrightarrow
\begin{tikzpicture}[xscale=.5,yscale=.33,baseline]
\draw[thick] (2,-1) to [out=90,in=-90] (0,2);
\draw[thick] (0,0) to [out=90,in=-90](2,1);
\draw[thick] (1,1) to [out=-90,in=90] (1.5,.5);

\draw[thick] (0,-1) to [out=90,in=-90] (1.5,.5);
\draw[thick] (1,-1) to [out=90,in=-90] (0,0);

\draw[thick] (1,1) to [out=90,in=-90] (2,2);
\draw[thick] (2,1) to [out=90,in=-90] (1,2);
\end{tikzpicture}
\ \dot{\sim}\ 
\begin{tikzpicture}[xscale=.5,yscale=.33,baseline]
\draw[thick] (0,0) to [out=90,in=-90](2,1);
\draw[thick] (1,0) to [out=90,in=-90] (.5,.5);
\draw[thick] (2,-1) to [out=90,in=-90] (0,2);

\draw[thick] (0,-1) to [out=90,in=-90] (1,0);
\draw[thick] (1,-1) to [out=90,in=-90] (0,0);

\draw[thick] (.5,.5) to [out=90,in=-90] (2,2);
\draw[thick] (2,1) to [out=90,in=-90] (1,2);
\end{tikzpicture}
\Leftrightarrow
\begin{tikzpicture}[xscale=.5,yscale=1,baseline=10]
\draw[thick] (0,0)--(2,1);
\draw[thick] (1,0) to [out=90,in=-90] (.5,.5);
\draw[thick] (1,1) to [out=-90,in=90] (.5,.5);
\draw[thick] (2,0)--(0,1);
\end{tikzpicture}
\ \dot{\sim}\ 
\begin{tikzpicture}[xscale=.5,yscale=1,baseline=10]
\draw[thick] (0,0) -- (2,1);
\draw[thick] (1,0) to [out=90,in=-90] (1.5,.5);
\draw[thick] (1,1) to [out=-90,in=90] (1.5,.5);
\draw[thick] (2,0)--(0,1);
\end{tikzpicture}
\end{equation}
and
\begin{equation}
\begin{tikzpicture}[xscale=.5,yscale=1,baseline=10]
\draw[thick] (0,0) -- (2,1);
\draw[thick] (1,0) to [out=90,in=-90] (1.5,.5);
\draw[thick] (1,1) to [out=-90,in=90] (1.5,.5);
\draw[thick] (2,0)--(0,1);
\end{tikzpicture}
\ \dot{\sim}\ 
\begin{tikzpicture}[xscale=.5,yscale=1,baseline=10]
\draw[thick] (0,0)--(2,1);
\draw[thick] (1,0) to [out=90,in=-90] (.5,.5);
\draw[thick] (1,1) to [out=-90,in=90] (.5,.5);
\draw[thick] (2,0)-- (0,1);
\end{tikzpicture}
\Leftrightarrow
\begin{tikzpicture}[xscale=.5,yscale=.33,baseline]
\draw[thick] (0,-1) to [out=90,in=-90] (2,2);
\draw[thick] (0,2) to [out=-90,in=90] (1.5,.5);
\draw[thick] (2,0)to [out=90,in=-90] (0,1);
\draw[thick] (0,1) to [out=90,in=-90] (1,2);
\draw[thick] (1,0) to [out=90,in=-90](1.5,.5);
\draw[thick] (1,-1) to [out=90,in=-90] (2,0);
\draw[thick] (2,-1) to [out=90,in=-90] (1,0);
\end{tikzpicture}
\ \dot{\sim}\ 
\begin{tikzpicture}[xscale=.5,yscale=.33,baseline]
\draw[thick] (0,-1) to [out=90,in=-90] (2,2);
\draw[thick] (1,1) to [out=-90,in=90] (.5,.5);
\draw[thick] (2,0) to [out=90,in=-90]  (0,1);
\draw[thick] (0,1) to [out=90,in=-90] (1,2);
\draw[thick] (1,1) to [out=90,in=-90] (0,2);
\draw[thick] (1,-1) to [out=90,in=-90] (2,0);
\draw[thick] (2,-1) to [out=90,in=-90](.5,.5);
\end{tikzpicture}
\Leftrightarrow
\begin{tikzpicture}[xscale=.5,yscale=1,baseline=10]
\draw[thick] (0,0)--(2,1);
\draw[thick] (1,0) to [out=90,in=-90] (.5,.5);
\draw[thick] (1,1) to [out=-90,in=90] (.5,.5);
\draw[thick] (2,0)-- (0,1);
\end{tikzpicture}
\ \dot{\sim}\ 
\begin{tikzpicture}[xscale=.5,yscale=1,baseline=10]
\draw[thick] (0,0) -- (2,1);
\draw[thick] (1,0) to [out=90,in=-90] (1.5,.5);
\draw[thick] (1,1) to [out=-90,in=90] (1.5,.5);
\draw[thick] (2,0)--(0,1);
\end{tikzpicture}
\end{equation}
which correspond to the respective equivalences
\begin{equation}
(\sigma,\varepsilon)\Leftrightarrow (\sigma,\varepsilon)\circ (1,2)\ \text{ and }\ (\sigma,\varepsilon)\Leftrightarrow (\sigma,\varepsilon)\circ (2,3)
\end{equation}
in the set $\operatorname{Sym}(3)\times \{\pm1\}^3$ where we interpret $(\sigma,\varepsilon)\in\operatorname{Sym}(3)\times \{\pm1\}^3$  as the map
\begin{equation}
 (\sigma,\varepsilon)\colon \{1,2,3\}\to \{1,2,3\}\times\{\pm1\},\quad i\mapsto (\sigma(i),\varepsilon_i).
\end{equation}
Thus, in conjunction with the equivalence~\eqref{eq:equiv(-1,.)}, the right action of the group $\operatorname{Sym}(3)$ on the set $\operatorname{Sym}(3)\times \{\pm1\}^3$ establishes the equivalence of any RIII move to the reference move~\eqref{eq.basriii} and thereby the invariance of $RT_r$ with respect to all RIII moves.

Finally, in order to prove invariance of $J_r$ with respect to all RI moves, we need to check only the invariance with respect four basic moves of the form
\begin{equation}\label{eq:basicRI}
\begin{tikzpicture}[xscale=.5,yscale=.5,baseline=5]
\draw[thick] (0,0) to [out=0,in=0] (1,1);
\draw[thick] (1,1) to [out=180,in=180] (2,0);
\end{tikzpicture}
\ \sim \ 
\begin{tikzpicture}[xscale=.5,yscale=.5,baseline=5]
\draw[thick] (0,0) to [out=0,in=180] (1,1);
\draw[thick] (1,1) to [out=0,in=180] (2,0);
\end{tikzpicture}
\end{equation}
as all others are consequences of the basic ones and the intermediate equivalence relation $\dot{\sim}$ generated by the moves R0$^\pm$, RII and RIII:
\begin{equation}
\begin{tikzpicture}[xscale=.5,yscale=.5,baseline=5]
\draw[thick] (0,0)-- (0,1);
\end{tikzpicture}
\ = \ 
\begin{tikzpicture}[xscale=.25,yscale=.5,baseline=5]
\draw[thick] (0,0) to [out=90,in=180] (1,1);
\draw[thick] (1,1) to [out=0,in=180] (2,0);
\draw[thick] (2,0) to [out=0,in=-90] (3,1);
\end{tikzpicture}
\ \sim\ 
\begin{tikzpicture}[xscale=.5,yscale=.5,baseline=5]
\draw[thick] (0,0) to [out=0,in=0] (1,1);
\draw[thick] (1,1) to [out=180,in=180] (2,0);
\draw[thick] (2,0) to [out=0,in=-90] (3,1);
\end{tikzpicture}
\ \dot{\sim}\ 
\begin{tikzpicture}[xscale=.5,yscale=.5,baseline=5]
\draw[thick] (0,.5) to [out=90,in=90] (1,0);
\draw[thick] (0,.5) to [out=-90,in=-90] (1,1);
\end{tikzpicture}
\ \dot{\sim}\ 
\begin{tikzpicture}[xscale=.5,yscale=.5,baseline=5]
\draw[thick] (0,1) to [out=0,in=0] (1,0);
\draw[thick] (1,0) to [out=180,in=180] (2,1);
\draw[thick] (2,1) to [out=0,in=90] (3,0);
\end{tikzpicture}
\Rightarrow
\begin{tikzpicture}[xscale=.5,yscale=.5,baseline=5]
\draw[thick] (0,1) to [out=0,in=180] (1,0);
\draw[thick] (1,0) to [out=0,in=180] (2,1);
\end{tikzpicture}
\ \sim\ 
\begin{tikzpicture}[xscale=.5,yscale=.5,baseline=5]
\draw[thick] (0,1) to [out=0,in=0] (1,0);
\draw[thick] (1,0) to [out=180,in=180] (2,1);
\end{tikzpicture}
\end{equation}
and
\begin{equation}
\begin{tikzpicture}[xscale=.5,yscale=.5,baseline=5]
\draw[thick] (0,0)-- (0,1);
\end{tikzpicture}
\ = \ 
\begin{tikzpicture}[xscale=.25,yscale=.5,baseline=5]
\draw[thick] (0,0) to [out=90,in=180] (1,1);
\draw[thick] (1,1) to [out=0,in=180] (2,0);
\draw[thick] (2,0) to [out=0,in=-90] (3,1);
\end{tikzpicture}
\ \sim\ 
\begin{tikzpicture}[xscale=.5,yscale=.5,baseline=5]
\draw[thick] (0,0) to [out=90,in=180] (1,1);
\draw[thick] (1,1) to [out=0,in=0] (2,0);
\draw[thick] (2,0) to [out=180,in=180] (3,1);
\end{tikzpicture}
\ \dot{\sim}\ 
\begin{tikzpicture}[xscale=.5,yscale=.5,baseline=5]
\draw[thick] (0,0) to [out=90,in=90] (1,.5);
\draw[thick] (0,1) to [out=-90,in=-90] (1,.5);
\end{tikzpicture}\ .
\end{equation}

Les us analyse the four cases of \eqref{eq:basicRI} separately.

Case~1. If diagrams $D$ and $D'$ differ by the fragments  
\begin{equation}
D\ni\begin{tikzpicture}[xscale=.5,yscale=.5,baseline=5]
\draw[thick,->] (1,1) to [out=180,in=180] (2,0);
\draw[line width=3,white] (0,0) to [out=0,in=0] (1,1);
\draw[thick] (0,0) to [out=0,in=0] (1,1);
\end{tikzpicture}
\ ,\quad 
\begin{tikzpicture}[xscale=.5,yscale=.5,baseline=5]
\draw[thick] (0,0) to [out=0,in=180] (1,1);
\draw[thick,->] (1,1) to [out=0,in=180] (2,0);
\end{tikzpicture}
\in D',
\end{equation}
then, by the definition of the normalisation of a long knot diagram, we have the equality $\dot D=\dot D'$. Thus, $J_r(D)=J_r(D')$.

Case~2. Diagrams $D$ and $D'$ differ by the fragments 
\begin{equation}
D\ni \begin{tikzpicture}[xscale=.5,yscale=.5,baseline=5]
\draw[thick] (0,0) to [out=0,in=0] (1,1);
\draw[line width=3,white] (1,1) to [out=180,in=180] (2,0);
\draw[thick,->] (1,1) to [out=180,in=180] (2,0);
\end{tikzpicture}
\ ,\quad
\begin{tikzpicture}[xscale=.5,yscale=.5,baseline=5]
\draw[thick] (0,0) to [out=0,in=180] (1,1);
\draw[thick,->] (1,1) to [out=0,in=180] (2,0);
\end{tikzpicture}
\in\ D'
\end{equation}
so that the normalised diagrams $\dot{D}$ and $\dot{D}'$ differ by the fragments
\begin{equation}
\dot{D}\ni \begin{tikzpicture}[xscale=.5,yscale=.5,baseline=5]
\draw[thick] (0,0) to [out=0,in=0] (1,1);
\draw[line width=3,white] (1,1) to [out=180,in=180] (2,0);
\draw[thick,->] (1,1) to [out=180,in=180] (2,0);
\end{tikzpicture}
\ ,\quad
\begin{tikzpicture}[xscale=.5,yscale=.5,baseline=5]
\draw[thick,->] (1,1) to [out=180,in=180] (2,0);
\draw[line width=3,white] (0,0) to [out=0,in=0] (1,1);
\draw[thick] (0,0) to [out=0,in=0] (1,1);
\end{tikzpicture}
\in\ \dot{D}'
\end{equation}
which imply that
\begin{equation}\label{eq:case2writhe}
w(\dot{D})=2+w(\dot{D}')\Rightarrow\xi^+ \circ\xi^{-w(\dot{D})/2}= \xi^{-w(\dot{D}')/2}.
\end{equation}
On the other hand, we have the equivalence
\begin{equation}
\dot{D}\ni \begin{tikzpicture}[xscale=.5,yscale=.5,baseline=5]
\draw[thick] (0,0) to [out=0,in=0] (1,1);
\draw[line width=3,white] (1,1) to [out=180,in=180] (2,0);
\draw[thick,->] (1,1) to [out=180,in=180] (2,0);
\end{tikzpicture}
\ \dot{\sim}\ 
\begin{tikzpicture}[xscale=.5,yscale=.5,baseline=5]
\draw[thick] (0,0) to [out=0,in=0] (1,1);
\draw[thick,->] (3,1) to [out=180,in=180] (4,0);
\draw[line width=3,white] (2,0) to [out=180,in=0] (3,1);
\draw[thick] (2,0) to [out=180,in=0] (3,1);
\draw[line width=3,white]  (1,1) to [out=180,in=0] (2,0);
\draw[thick] (1,1) to [out=180,in=0] (2,0);
\end{tikzpicture}
\in \dot{D}'\circ\xi^+
\end{equation}
which, together with \eqref{eq:case2writhe}, implies that
\begin{equation}
 \dot{D}\circ \xi^{-w(\dot{D})/2}\  \dot{\sim}\ \dot{D}'\circ\xi^+\circ \xi^{-w(\dot{D})/2}= \dot{D}'\circ \xi^{-w(\dot{D}')/2}\Rightarrow J_r(D)=J_r(D').
\end{equation}

Case~3. Diagrams $D$ and $D'$ differ by the fragments 
\begin{equation}
D\ni \begin{tikzpicture}[xscale=.5,yscale=.5,baseline=5]
\draw[thick,<-] (0,0) to [out=0,in=0] (1,1);
\draw[line width=3,white] (1,1) to [out=180,in=180] (2,0);
\draw[thick] (1,1) to [out=180,in=180] (2,0);
\end{tikzpicture}
\ ,\quad
\begin{tikzpicture}[xscale=.5,yscale=.5,baseline=5]
\draw[thick,<-] (0,0) to [out=0,in=180] (1,1);
\draw[thick] (1,1) to [out=0,in=180] (2,0);
\end{tikzpicture}
\in\ D'
\end{equation}
so that 
\begin{equation}
\dot D\ni \begin{tikzpicture}[xscale=.5,yscale=.5,baseline=5]
\draw[thick,<-] (0,0) to [out=0,in=-90] (1.5,1);
\draw[thick] (1,2) to [out=180,in=90] (1.5,1);
\draw[line width=3,white] (2,0) to [out=180,in=-90] (.5,1);
\draw[line width=3,white](1,2) to [out=0,in=90] (.5,1);
\draw[thick] (2,0) to [out=180,in=-90] (.5,1);
\draw[thick] (1,2) to [out=0,in=90] (.5,1);
\end{tikzpicture}
\ \dot{\sim}\ 
\begin{tikzpicture}[xscale=.5,yscale=.5,baseline=5]
\draw[thick,<-] (0,0) to [out=0,in=180] (1,1);
\draw[thick] (1,1) to [out=0,in=180] (2,0);
\end{tikzpicture}
\in\ \dot D'
\Rightarrow J_r(D)=J_r(D').
\end{equation}

Case~4. Diagrams $D$ and $D'$ differ by the fragments 
\begin{equation}
D\ni\begin{tikzpicture}[xscale=.5,yscale=.5,baseline=5]
\draw[thick] (1,1) to [out=180,in=180] (2,0);
\draw[line width=3,white] (0,0) to [out=0,in=0] (1,1);
\draw[thick,<-] (0,0) to [out=0,in=0] (1,1);
\end{tikzpicture}
\ ,\quad 
\begin{tikzpicture}[xscale=.5,yscale=.5,baseline=5]
\draw[thick,<-] (0,0) to [out=0,in=180] (1,1);
\draw[thick] (1,1) to [out=0,in=180] (2,0);
\end{tikzpicture}
\in D'
\end{equation}
so that we have for the corresponding normalised diagrams
\begin{equation}
\dot D\ni \begin{tikzpicture}[xscale=.5,yscale=.5,baseline=5]
\draw[thick] (2,0) to [out=180,in=-90] (.5,1);
\draw[thick] (1,2) to [out=180,in=90] (1.5,1);
\draw[line width=3,white] (0,0) to [out=0,in=-90] (1.5,1);
\draw[line width=3,white](1,2) to [out=0,in=90] (.5,1);
\draw[thick,<-] (0,0) to [out=0,in=-90] (1.5,1);
\draw[thick] (1,2) to [out=0,in=90] (.5,1);
\end{tikzpicture}
\ \dot\sim\ 
\begin{tikzpicture}[xscale=.5,yscale=.5,baseline=5]
\draw[thick,<-] (-2,0) to [out=0,in=180] (-1,1);
\draw[thick] (-1,1) to [out=0,in=180] (0,0);
\draw[thick] (2,0) to [out=180,in=-90] (.5,1);
\draw[thick] (1,2) to [out=180,in=90] (1.5,1);
\draw[line width=3,white] (0,0) to [out=0,in=-90] (1.5,1);
\draw[line width=3,white](1,2) to [out=0,in=90] (.5,1);
\draw[thick] (0,0) to [out=0,in=-90] (1.5,1);
\draw[thick] (1,2) to [out=0,in=90] (.5,1);
\end{tikzpicture}
\in  \dot D'\circ \xi^- 
\end{equation}
and
\begin{equation}
w(\dot{D})=w(\dot{D}')-2\Rightarrow\xi^- \circ\xi^{-w(\dot{D})/2}= \xi^{-w(\dot{D}')/2}.
\end{equation}
Thus,
\begin{equation}
\dot D\circ\xi^{-w(\dot{D})/2}\ \dot\sim\  \dot D'\circ\xi^-\circ \xi^{-w(\dot{D})/2}=\dot D'\circ\xi^{-w(\dot{D}')/2}\Rightarrow  J_r(D)=J_r(D').
\end{equation}

\end{proof}
\section{Rigid R-matrices from racks}\label{sec:rrmfr}

A \emph{\color{blue} binary relation} from a set $X$ to a set $Y$ as a subset of the cartesian product $X\times Y$. The composition of two binary relations $R\subset X\times Y$ and $S\subset Y\times Z$ is the binary relation  $S\circ R\subset  X\times Z$ defined by
\begin{equation}
S\circ R:=\{ (x,z)\in X\times Z\mid \exists y\in Y\colon\ (x,y)\in R,\  (y,z)\in S\}.
\end{equation}

A  \emph{\color{blue} span} from a set $X$ to a set $Y$ is a triple $U=(U,\operatorname{s}_U,\operatorname{t}_U)$ where $U$ is a set and $\operatorname{s}_U\colon U\to X$ and $\operatorname{t}_U\colon U\to Y$ are set theoretical maps. The composition of two spans $U$ from $X$ to $Y$ and $V$ from $Y$ to $Z$ is the span from $X$ to $Z$ defined as the pullback  space (fibered product) $V\circ U:=U\times_Y V$ together with the natural projections to $X$ and $Z$. 

Two spans $U$ and $V$ from $X$ to $Y$ are called \emph{\color{blue} equivalent} if there exists a bijection $f\colon U\to V$ such that  $\operatorname{s}_V\circ f=\operatorname{s}_U$ and $\operatorname{t}_V\circ f=\operatorname{t}_U$. The composition of spans induces an associative binary operation for the equivalence classes of spans. 

Any binary relation $R\subset X\times Y$ is a special case of a span with $\operatorname{s}_R\colon R\to X$ and $\operatorname{t}_R\colon R\to Y$ being the canonical projections.

Let $\mathbf{Set}$ be the monoidal category of sets with the cartesian product as the monoidal product, and  $\mathbf{Rel}$ (respectively $\mathbf{Span}$) the extension of  $\mathbf{Set}$  with morphisms given by binary relations (respectively equivalence classes of spans).  For a morphism $Z\colon X\to Y$ in $\mathbf{Span}$, and any $(x,y)\in X\times Y$, we denote
\begin{equation}
Z(x,y):=\operatorname{s}_Z^{-1}(x)\cap\operatorname{t}_Z^{-1}(y).
\end{equation}
We have a canonical monoidal functor 
\begin{equation}
\varpi\colon\mathbf{Span}\to \mathbf{Rel}
\end{equation}
which is identity on the level of objects and for any morphism $Z\colon X\to Y$ in $\mathbf{Span}$, the corresponding morphism in $ \mathbf{Rel}$ is given by 
\begin{equation}
\varpi(Z)=\{(x,y)\in X\times Y\mid Z(x,y)\ne\emptyset\}.
\end{equation}
Notice that if $Z$ is a relation (as a particular case of spans) then $\varpi(Z)=Z$.

Given  a set theoretical map $f\colon X\to Y$, its  graph
\begin{equation}
\Gamma_f:=\{(x,f(x))\mid x\in X\}\subset X\times Y
\end{equation}
is naturally interpreted as a morphism $\rho(f)$ in  $\mathbf{Rel}$ and a morphism $\sigma(f)$ in $\mathbf{Span}$. 

The advantage of the categories $\mathbf{Rel}$ and $\mathbf{Span}$ over $\mathbf{Set}$ is their rigidity, namely, for any set $X$, the diagonal
$
\Delta_X:=\Gamma_{\operatorname{id}_X},
$
interpreted as morphisms $\varepsilon_X\colon X\times X\to \{0\}$ and $\eta_X\colon \{0\}\to X\times X$  in $\mathbf{Rel}$ and $\mathbf{Span}$, gives rise 
to a canonical adjunction $(X,X,\varepsilon_X,\eta_X)$ both in $\mathbf{Rel}$ and $\mathbf{Span}$.
 
Let $X$ be a  (left) \emph{\color{blue} rack}~\cite{MR1194995,MR0021002,MR638121,MR672410}, that is a set with a map
$$
X^2\to X^2,\quad (x,y)\mapsto (x\cdot y, x*y),
$$
such that the binary operation $x\cdot y$ is left self-distributive
$$
x\cdot(y\cdot z)=(x\cdot y)\cdot (x\cdot z),\quad \forall (x,y,z)\in X^3,
$$
and 
$$
x*(x\cdot y)=y,\quad \forall (x,y)\in X^2.
$$
It is easily verified that for any rack $X$, the set-theoretical map 
$$
r\colon X^2\to X^2,\quad (x,y)\mapsto (x\cdot y,x),
$$
is a rigid R-matrix in the categories $\mathbf{Rel}$ and $\mathbf{Span}$. 
Moreover, all the relevant morphisms  are realised by set-theoretical maps.
Indeed, introducing the maps
$$
r',s,s'\colon X^2\to X^2,\quad r'(x,y)=(y, y\cdot x),\ s(x,y)=(x*y,x),\ s'(x,y)=(y,y*x),
$$
we obtain
$$
r^{-1}=\widetilde{r}=s', \quad (\widetilde{r})^{-1}=r,\quad
 \widetilde{r^{-1}}=\widetilde{\widetilde{r}}=s,\quad \left(\widetilde{r^{-1}}\right)^{-1}=\widetilde{\widetilde{r^{-1}}}=r'.
$$

Thus, we obtain two long knot invariants $J_{\rho(r)}(D)$ in $\mathbf{Rel}$ and
$J_{\sigma(r)}(D)$ in $\mathbf{Span}$ which are related to each other by the equality
\begin{equation}
J_{\rho(r)}(D)=\varpi(J_{\sigma(r)}(D)).
\end{equation}

\subsection{Racks associated to pointed groups} Let $(G,\mu)$ be a \emph{\color{blue} pointed group} that is a group $G$ together with a fixed element $\mu\in G$. Then, it is easily verified that the set $G$ with the map
\begin{equation}
G\times G\to G\times G, \quad (g,h)\mapsto (g\mu g^{-1}h,g\mu^{-1} g^{-1}h)
\end{equation}
is a rack, and, thus, it gives rise to a rigid R-matrix 
\begin{equation}
r_{G,\mu}\colon G\times G\to G\times G,\quad (g,h)\mapsto (g\mu g^{-1}h,g),
\end{equation}
 both in  the categories $\mathbf{Rel}$ and $\mathbf{Span}$. In this way, we obtain two long knot invariants $J_{\rho(r_{G,\mu})}(D)$ and $J_{\sigma(r_{G,\mu})}(D)$ related to each other by the equality
 \begin{equation}
J_{\rho(r_{G,\mu})}(D)=\varpi(J_{\sigma(r_{G,\mu})}(D)).
\end{equation}
\begin{theorem}\label{thm:1}
 There exists a canonical choice of  a meridian-longitude pair  $(m,\ell)$ of long knots such that  the set $(J_{\sigma(r_{G,\mu})}(D))(1,\lambda)$ is in bijection with the set of group homomorphisms
 \begin{equation}\label{eq:gr-hom}
\{h\colon \pi_1(\R^3\setminus f(\R),x_0)\to G\mid h(m)=\mu,\ h(\ell)=\lambda\}
\end{equation}
where $f\colon\R\to\R^3$ is a long knot represented by $D$.
\end{theorem}
\begin{proof}
Let $f\colon \R\to\R^3$ be a long knot whose image under the projection 
\begin{equation}
p\colon\R^3\to\R^2,\quad (x,y,z)\mapsto (y,z).
\end{equation}
 is the diagram $\tilde D:=\dot{D}\circ \xi^{-w(\dot{D})/2}$ with linearly ordered (from bottom to top) set of arcs $a_0,a_1,\dots, a_n$. As a result, the set of crossings acquires a linear order as well $\{c_i\mid 1\le i\le n\}$ where $c_i$ is the crossing separating the arcs $a_{i-1}$ and $a_i$ and with the over passing arc $a_{\kappa_i}$ for a  uniquely defined map 
\begin{equation}
\kappa\colon \{1,\dots,n\}\to\{0,\dots,n\}.
\end{equation}

 Let $t_0,t_1,\dots,t_n\in\R$ be a strictly increasing sequence such that $f(t)=(0,0,t)$ for all $t\not \in [t_0,t_n]$, and for each $i\in\{1,\dots, n-1\}$, $p(f(t_i))$ belongs to arc $a_i$ and is distinct from any crossing. Choose a base point $x_0=(s,0,0)$ with sufficiently large $s\in\R_{>0}$, a sufficiently small $\epsilon\in\R_{>0}$,  and define the following paths
 \begin{multline}
\alpha_0,\beta_i,\gamma_i\colon [0,1]\to \R^3,\ i\in\{0,\dots, n\},\quad \alpha_0(t)=(\epsilon\cos(2\pi t),-\epsilon\sin(2\pi t),t_0),\\
 \beta_i(t)=(1-t)x_0+(f(t_i)+(\epsilon,0,0))t,\quad
 \gamma_i(t)=f((1-t)t_i+t t_0)+(\epsilon,0,0).
\end{multline}

 To each arc $a_i$ of $\tilde D$, we associate the homotopy class
 \begin{equation}
e_i:=[\beta_i\cdot\gamma_i\cdot \bar\beta_0]\in  \pi_1(\R^3\setminus f(\R),x_0),
\end{equation}
so that $e_0=1$, and the Wirtinger generator  
\begin{equation}
w_i:=[\beta_i\cdot\gamma_i\cdot\alpha_0\cdot\bar\gamma_i\cdot\bar \beta_i] \in\pi_1(\R^3\setminus f(\R),x_0).
\end{equation}
We have the equalities
\begin{equation}\label{eq:wi=eiw0ei}
w_i=e_i w_0 e_i^{-1},\quad \forall i\in\{0,1,\dots,n\},
\end{equation}
\begin{equation}\label{eq:ei=wkiei-1}
e_i=w_{\kappa_i}^{\varepsilon_i}e_{i-1},\quad  \forall i\in\{1,\dots,n\},
\end{equation}
where $\varepsilon_i\in\{\pm1\}$ is the sign of the crossing $c_i$,
and 
\begin{equation}
e_i=w_{\kappa_i}^{\varepsilon_i}w_{\kappa_{i-1}}^{\varepsilon_{i-1}}\cdots w_{\kappa_1}^{\varepsilon_1},\quad \forall i\in\{1,\dots,n\}.
\end{equation}

We define the canonical meridian-longitude pair $(m,\ell)$ as follows
 \begin{equation}
m:=w_0,\quad \ell:=e_n.
\end{equation}
Taking into account the condition $\sum_{i=1}^n\varepsilon_i=0$, we see that $\ell$ has the trivial image in  $H_1(\R^3\setminus f(\R),\Z)$.

Let us show that the following finitely presented groups are isomorphic to the knot group $\pi_1(\R^3\setminus f(\R),x_0)$:
\begin{equation}
E:=\langle m,e_0,\dots, e_n\mid e_0=1,\ e_i=e_{\kappa_i}m^{\varepsilon_i} e_{\kappa_i}^{-1} e_{i-1},\ 1\le i\le n\rangle
\end{equation}
and 
\begin{equation}
W:=\langle w_0,\dots,w_n\mid w_{\kappa_i}^{\varepsilon_i}w_{i-1}=w_iw_{\kappa_i}^{\varepsilon_i},\ 1\le i\le n\rangle.
\end{equation}
As $W$ is nothing else but the Wirtinger presentation of  $\pi_1(\R^3\setminus f(\R),x_0)$, it suffices to see the isomorphism $E\simeq W$. To see the latter,
 we remark that there are two group homomorphisms 
\begin{equation}
u\colon W\to E,\quad  w_i\mapsto e_i me_i^{-1},\quad  i\in\{0,1,\dots,n\},
\end{equation}
and
\begin{equation}
v\colon E\to W,\quad  m\mapsto w_0,\quad e_0\mapsto 1,\quad e_i\mapsto w_{\kappa_i}^{\varepsilon_i}w_{\kappa_{i-1}}^{\varepsilon_{i-1}}\cdots w_{\kappa_1}^{\varepsilon_1},\quad \forall i\in\{1,\dots,n\}.
\end{equation}
Indeed, we have
\begin{multline}
 u(w_{\kappa_i})^{\varepsilon_i}u(w_{i-1})=u(w_i)u(w_{\kappa_i})^{\varepsilon_i}
 \Leftrightarrow  u(w_{\kappa_i})^{\varepsilon_i}e_{i-1}me_{i-1}^{-1}=e_ime_i^{-1} u(w_{\kappa_i})^{\varepsilon_i}\\
 \Leftrightarrow  e_i^{-1}u(w_{\kappa_i})^{\varepsilon_i}e_{i-1}m=me_i^{-1} u(w_{\kappa_i})^{\varepsilon_i}e_{i-1}\Leftarrow e_i^{-1} u(w_{\kappa_i})^{\varepsilon_i}e_{i-1}=1\\
 \Leftrightarrow   e_i=u(w_{\kappa_i})^{\varepsilon_i}e_{i-1} \Leftrightarrow e_i=e_{\kappa_i}m^{\varepsilon_i} e_{\kappa_i}^{-1} e_{i-1}
\end{multline}
implying that $u$ is a group homomorphism. We also have
\begin{multline}
v(e_i)=v(e_{\kappa_i})v(m)^{\varepsilon_i} v(e_{\kappa_i})^{-1} v(e_{i-1})\Leftrightarrow v(e_i) v(e_{i-1})^{-1}=v(e_{\kappa_i})w_0^{\varepsilon_i} v(e_{\kappa_i})^{-1}\\ 
\Leftrightarrow w_{\kappa_i}=v(e_{\kappa_i})w_0 v(e_{\kappa_i})^{-1}\Leftarrow \{w_i=v(e_i)w_0v(e_i)^{-1}\}_{0\le i\le n}
\end{multline}
and, for all $i\in\{1,\dots,n\}$,
\begin{multline}
v(e_i)w_0=w_{\kappa_i}^{\varepsilon_i}\cdots w_{\kappa_1}^{\varepsilon_1}w_0=w_{\kappa_i}^{\varepsilon_i}\cdots w_{\kappa_2}^{\varepsilon_2}w_1w_{\kappa_1}^{\varepsilon_1}
=\dots\\=w_{\kappa_i}^{\varepsilon_i}\cdots w_{\kappa_k}^{\varepsilon_k}w_{k-1}w_{\kappa_{k-1}}^{\varepsilon_{k-1}}\cdots w_{\kappa_1}^{\varepsilon_1}=\dots=
w_iw_{\kappa_i}^{\varepsilon_i}\cdots w_{\kappa_1}^{\varepsilon_1}=w_iv(e_i)
\end{multline}
implying that $v$ is a group homomorphism as well.

It remains to show that $v\circ u=\operatorname{id}_W$ and  $u\circ v=\operatorname{id}_E$. Indeed, for all $i\in\{0,\dots,n\}$, we have
\begin{equation}
 v(u(w_i))=v(e_i me_i^{-1})=v(e_i)v(m)v(e_i)^{-1}=v(e_i)w_0v(e_i)^{-1}=w_i
\end{equation}
implying that  $v\circ u=\operatorname{id}_W$.  We prove that $u(v(e_i))=e_i$ for all $i\in\{0,\dots,n\}$ by recursion on $i$. For $i=0$, we have
\begin{equation}
 u(v(e_0))=u(1)=1=e_0.
\end{equation}
Assuming that $u(v(e_{k-1}))=e_{k-1}$ for some $k\in\{1,\dots,n-1\}$, we calculate
\begin{equation}
 u(v(e_{k}))=u(w_{\kappa_k}^{\varepsilon_k}v(e_{k-1}))=u(w_{\kappa_k})^{\varepsilon_k}u(v(e_{k-1}))=e_{\kappa_k}m^{\varepsilon_k}e_{\kappa_k}^{-1}e_{k-1}=e_k.
\end{equation}
 
Now, any element $g$ of the set  $(J_{\sigma(r_{G,\mu})}(D))(1,\lambda)$ is a map
\begin{equation}
g\colon \{0,1,\dots,n\}\to G
\end{equation}
such that
\begin{equation}
g_0=1,\ g_n=\lambda,\ g_i=g_{\kappa_i}\mu^{\varepsilon_i} g_{\kappa_i}^{-1} g_{i-1},\quad \forall i\in\{1,\dots,n\}.
\end{equation}
That means that $g$ determines a unique group homomorphism
\begin{equation}
h_g\colon E\to G
\end{equation}
such that $h_g(m)=\mu$ and $h_g(e_i)=g_i$ for all $i\in\{0,1,\dots,n\}$. On the other hand, any group homomorphism $h\colon E\to G$ such that $h(m)=\mu$ and $h(\ell)=\lambda$ is of the form $h=h_g$ where $g_i=h(e_i)$. Thus, the map $g\mapsto h_g$ is a set-theoretical bijection between $(J_{\sigma(r_{G,\mu})}(D))(1,\lambda)$ and the set of group homomorphisms~\eqref{eq:gr-hom}.
\end{proof}
\begin{remark}
 Theorem~\ref{thm:1} illustrates the importance of considering long knots as opposed to closed knots. Namely, by closing a long knot, one identifies two open strands and all the associated data. In particular, one has to impose the equality $\lambda=1$ that corresponds to considering only those representations where the longitude is realized trivially. That means that in the case of closed diagrams one would obtain less powerful invariants.
\end{remark}
 \subsection{An extended Heisenberg group}
Let $R$ be a commutative unital ring and $G$, the subgroup of $GL(3,R)$ given by upper triangular matrices of the form 
\begin{equation}
\begin{pmatrix}
 a&b&d\\
 0&1&c\\
 0&0&a
\end{pmatrix},\quad a\in U(R),\ (b,c,d)\in R^3,
\end{equation}
where $U(R)$ the group of units of $R$. Then, the elements
\begin{equation}
\mu=\begin{pmatrix}
 t&0&0\\
 0&1&0\\
 0&0&t
\end{pmatrix},\quad 
\lambda=\begin{pmatrix}
 1&0&s\\
 0&1&0\\
 0&0&1
\end{pmatrix},
\end{equation}
commute with each other for any $t\in U(R)$ and $s\in R$. Given the fact that $G$ is an algebraic group, the invariant $(J_{\rho(r_{G,\mu})}(D))(1,\lambda)$ factorises through
an ideal $I_D$ in the polynomial algebra $\Q[t,t^{-1},s]$. Examples of calculations show that this ideal is often principal and is generated by the polynomial $\Delta_Ds$ where $\Delta_D:=\Delta_D(t)$ is the Alexander polynomial of $D$. Nonetheless, there are examples for which this is not the case, at least if the Alexander polynomial has multiple roots. In Table~\ref{fig:1}, we have collected few selected examples of calculations for knots where the Alexander polynomials of the first three examples of knots $3_1$, $4_1$ and $6_2$ have simple roots, while for all other examples the respective Alexander polynomials have multiple roots and they are factorized into products of the Alexander polynomials of the first three examples. In particular, one can see that the knots $8_{10}$ and $8_{20}$ have different Alexander polynomials but one and the same ideal, while
the knots $8_{18}$ and $9_{24}$ have one and the same Alexander polynomial but different ideals. 
\begin{table}[h]
 \[
\begin{array}{c|c|c}
  \text{knot}& \Delta_D& I_D  \\
  \hline
   3_{1}&1-t+t^2   &  (\Delta_{3_1}s) \\
    4_{1}&1-3t+t^2   &  (\Delta_{4_1}s) \\
     6_{2}&1 - 3 t + 3 t^2 - 3 t^3 + t^4 &  (\Delta_{6_2}s) \\
  8_{10}&\Delta_{3_1}^3   &  (\Delta_{3_1}s,s^2) \\
 8_ {18}&\Delta_{4_1} \Delta_{3_1}^2   &   (\Delta_{4_1} \Delta_{3_1}s)\\
 8_{20}&\Delta_{3_1}^2&  (\Delta_{3_1}s,s^2)\\
  9_{24}&\Delta_{4_1} \Delta_{3_1}^2&  (\Delta_{4_1} \Delta_{3_1}s,\Delta_{4_1} s^2)\\
   10_{99}&\Delta_{3_1}^4   &  (\Delta_{3_1}s,s^2) \\
   10_{123}&\Delta_{6_2}^2   &  (\Delta_{6_2}s) \\
   10_{137}&\Delta_{4_1}^2   &  (\Delta_{4_1}s,s^2) \\
 11a_{5}&\Delta_{4_1}^3&  (\Delta_{4_1}s,s^2)
\end{array}
\]
\caption{Examples of calculation with the extended Heisenberg group.}\label{fig:1}
\end{table}

\section{Invariants from Hopf algebras}\label{sec:ifha}
Let $\mathbf{Hopf}_\K$ be the category of Hopf algebras over a field $\K$ with invertible antipode. We have  a contravariant endofunctor $(\cdot)^o\colon \mathbf{Hopf}_\K\to \mathbf{Hopf}_\K$
that associates to a Hopf algebra $H$ with the multiplication $\nabla$ its 
 \emph{\color{blue} restricted dual} 
\begin{equation}
H^o:=(\nabla^*)^{-1}(H^*\otimes H^*)
\end{equation}
that can also be identified with the vector subspace of the algebraic dual $H^*$ generated by all matrix coefficients of all finite dimensional representations of $H$ \cite{MR1786197}.

Following~\cite{MR1381692}, Drinfeld's   \emph{\color{blue} quantum double}  of $H\in \operatorname{Ob}\mathbf{Hopf}_\K$ is a Hopf algebra $D(H)\in \operatorname{Ob}\mathbf{Hopf}_\K$ uniquely determined by the property that there are two Hopf algebra inclusions
\begin{equation}
\imath\colon H\to D(H),\quad \jmath\colon H^{o,\text{op}}\to D(H)
\end{equation}
such that $D(H)$ is generated by their images subject to the commutation relations
\begin{equation}\label{eq:comm-rel-j-i}
\jmath(f)\imath(x)= \langle f_{(1)},x_{(1)}\rangle  \langle f_{(3)},S(x_{(3)})\rangle\imath(x_{(2)})\jmath(f_{(2)}),\quad \forall (x,f)\in H\times H^o,
\end{equation}
where we use Sweedler's notation for the comultiplication
\begin{equation}
\Delta(x)=x_{(1)}\otimes x_{(2)},\quad (\Delta\otimes\operatorname{id})(\Delta(x))=x_{(1)}\otimes x_{(2)}\otimes x_{(3)},\ \dots
\end{equation}

The restricted dual of the quantum double $(D(H))^o$ is a   \emph{\color{blue} dual quasitriangular} Hopf algebra  with the  \emph{\color{blue} dual universal  $R$-matrix}
\begin{equation}
\varrho\colon (D(H))^o\otimes (D(H))^o\to \K,\quad x\otimes y\mapsto \langle x,\jmath(\imath^o(y))\rangle
\end{equation}
which, among other things, satisfies the Yang--Baxter relation
\begin{equation}
\varrho_{1,2}*\varrho_{1,3}*\varrho_{2,3}=\varrho_{2,3}*\varrho_{1,3}*\varrho_{1,2}
\end{equation}
 in the convolution algebra $(((D(H))^o)^{\otimes3})^*$, and any finite-dimensional right comodule
\begin{equation}
V\to V\otimes (D(H))^o,\quad v\mapsto v_{(0)}\otimes v_{(1)},
\end{equation}
gives rise to a rigid $R$-matrix 
\begin{equation}
r_V\colon V\otimes V\to V\otimes V,\quad u\otimes v\mapsto v_{(0)}\otimes u_{(0)}\langle \varrho,u_{(1)}\otimes v_{(1)}\rangle.
\end{equation}
This implies that there exists a \emph{\color{blue} universal invariant} of long knots $Z_H(K)$ taking its values in the 
convolution algebra $((D(H))^o)^*$ such that 
\begin{equation}
J_{r_V}(K)v=v_{(0)}\langle Z_H(K),v_{(1)}\rangle,\quad \forall v\in V.
\end{equation}
\begin{remark}
By using the coend of the monoidal braided category of finite dimensional comodules over $ (D(H))^o$, the universal invariant $Z_H(K)$ can be interpreted via the Lyubashenko theory \cite{MR1324033}.  In the case of finite dimensional quasitriangular Hopf algebras, this coend approach is further developed by Virelizier in \cite{MR2251160}. 
\end{remark}
\begin{remark}
If $H$ is a finite-dimensional Hopf algebra with a linear basis $\{e_i\}_{i\in I}$ and $\{e^i\}_{i\in I}$ is the dual linear basis of $H^*$, then, the dual universal $R$-matrix is conjugate  to the universal $R$-matrix
 \begin{equation}\label{eq.un-R-mat}
R:=\sum_{i\in I}\jmath(e^i)\otimes\imath(e_i)\in D(H)\otimes D(H)
\end{equation}
in the sense that, for any $x,y\in (D(H))^o=(D(H))^*$, we have
\begin{multline}
\langle x\otimes y,R\rangle=\sum_{i\in I}\langle x,\jmath(e^i)\rangle\langle y,\imath(e_i)\rangle=\sum_{i\in I}\langle x,\jmath(e^i)\rangle\langle \imath^o(y),e_i\rangle\\
=\left\langle x,\jmath\left(\sum_{i\in I}\langle \imath^o(y),e_i\rangle e^i\right)\right\rangle=\left\langle x,\jmath\left(\imath^o(y)\right)\right\rangle=\langle \varrho, x\otimes y\rangle.
\end{multline}

In the infinite-dimensional case, formula~\eqref{eq.un-R-mat} is formal but it is a convenient and useful tool for actual calculations.
\end{remark}
\begin{remark}
 The algebra inclusion $D(H)\subset ((D(H))^o)^*$ allows to think of the convolution algebra  $((D(H))^o)^*$  as a certain algebra completion of the quantum double $D(H)$.
\end{remark}
\subsection{A Hopf algebra associated to a two-dimensional Lie group}

 Let $B$ be the commutative Hopf algebra over $\C$ generated by an invertible group-like element $a$ and element $b$ with the coproduct 
 \begin{equation}\label{eq:coprod-b}
\Delta(b)=a\otimes b+b\otimes 1
\end{equation}
so that the set of monomials $\{b^ma^n\mid (m,n)\in\Z_{\ge0}\times \Z\}$ is a linear basis of $B$.

The restricted dual Hopf algebra $B^o$ is generated by two primitive elements $\psi$ and $\phi$ and group-like elements 
\begin{equation}
u^\psi,\ e^{v\phi},\quad (u,v)\in \C_{\ne0}\times\C,
\end{equation}
which satisfy the commutation relations
 \begin{multline}
 [\psi,\phi]=\phi,\quad[\psi,e^{v\phi}]=v\phi e^{v\phi},\quad u^\psi \phi u^{-\psi}=u\phi,\quad u^\psi e^{v\phi}u^{-\psi}=e^{uv\phi},\\
 [\psi, u^\psi]=[\phi, e^{v\phi}]=0,\quad u^\psi z^\psi=(uz)^\psi,\quad e^{v\phi}e^{w\phi}=e^{(v+w)\phi}.
\end{multline}
As linear forms on $B$, they are defined by the relations
\begin{multline}
\langle\psi,b^ma^n\rangle=\delta_{m,0}n,\quad \langle\phi,b^ma^n\rangle=\delta_{m,1},\\
\langle u^\psi,b^ma^n\rangle=\delta_{m,0}u^n,\quad \langle e^{v\phi},b^ma^n\rangle=v^m,
\quad \forall(m,n)\in\Z_{\ge0}\times \Z.
\end{multline}

Using the notation $\dot{x}:=\imath(x)$ for $x\in\{a,b\}$, $\dot{y}:=\jmath(y)$ for $y\in\{\psi,\phi\}$ and
\begin{equation}
 u^{\dot{\psi}}:=\jmath(u^{\psi}),\quad e^{v\dot{\phi}}:=\jmath(e^{v\psi}),
\end{equation}
the commutation relations~\eqref{eq:comm-rel-j-i} in the case of the quantum double $D(B)$ take the form
\begin{equation}
[\dot{\psi},\dot{b}]=\dot{b},\quad
[\dot{\phi},\dot{b}]=1-\dot{a},\quad u^{\dot{\psi}}\dot{b}u^{-\dot{\psi}} =u\dot{b},\quad
e^{v\dot{\phi}}\dot{b} e^{-v\dot{\phi}}=\dot{b}+(1-\dot{a})v
\end{equation}
and $\dot{a}$ is central. The formal universal $R$-matrix that reads as
\begin{equation}\label{eq:un-r-mat}
R:=(1\otimes\dot{a})^{\dot{\psi}\otimes 1}e^{\dot{\phi}\otimes\dot{b}}=\sum_{m,n\ge0}\frac 1{n!}\binom{\dot{\psi}}{m}\dot{\phi}^n\otimes (\dot{a}-1)^m \dot{b}^n
\end{equation}
should be interpreted as follows. 

Any finite dimensional right comodule $V$ over $(D(B))^o$ is a left module over $D(B)$
defined by
\begin{equation}
xv=v_{(0)}\langle v_{(1)},x\rangle,\quad \forall (x,v)\in D(B)\times V.
\end{equation}
Thus, it suffices to make sense of formula~\eqref{eq:un-r-mat} in the case of an arbirary finite-dimensional representation of $D(B)$ where the elements $1-\dot{a}$, $\dot{b}$ and $\dot{\phi}$ are necessarily  nilpotent, so that the formal infinite double sum
truncates to a well defined finite sum.
\begin{conjecture}
 The universal invariant associated to the Hopf algebra $B$ is of the form
$
Z_B(K)=(\Delta_K(\dot{a}))^{-1}
$
 where $\Delta_K(t)$ is the Alexander polynomial of $K$ normalised so that $\Delta_K(1)=1$ and $\Delta_K(t)=\Delta_K(1/t)$.
\end{conjecture}
By direct computation, we were able to prove this conjecture in the case of the trefoil knot. Justification in the general case comes from the following reasoning.

The Hopf algebra $B$ can be $q$-deformed to a non-commutative Hopf algebra $B_q$ with the same coalgebra structure~\eqref{eq:coprod-b} but with $q$-comutative relation $ab=qba$. We have $B=B_1$. For $q$ not a root of unity, the quantum double $D(B_q)$ is closely related to the quantum group $U_q(sl_2)$. In particular, for each $n\in\Z_{>0}$, it admits an $n$-dimensional irreducible representation corresponding to the $n$-th colored Jones polynomial. In the limit  $n\to\infty$ with $q=t^{1/n}$ and fixed $t$, one recovers an infinite-dimensional representation of the Hopf algebra $B$ where the central element $\dot{a}$ takes the value $t$. On the other hand, according to the Melvin--Morton--Rozansky conjecture proven by Bar-Nathan and Garoufalidis in \cite{MR1389962} and  by Garoufalidis and L\^e in \cite{MR2860990}, the  $n$-th colored Jones polynomial in that limit tends to $(\Delta_K(t))^{-1}$.

\def\cprime{$'$} \def\cprime{$'$}

\end{document}